\documentclass[]{amsart}

\usepackage[utf8]{inputenc}
\usepackage[OT2,T1]{fontenc}
\DeclareSymbolFont{cyrletters}{OT2}{wncyr}{m}{n}
\DeclareMathSymbol{\Sha}{\mathalpha}{cyrletters}{"58}
\usepackage{graphicx}
\graphicspath{ {./images/} }

\usepackage[hyphens,spaces,obeyspaces]{url}
\usepackage[colorlinks,allcolors=blue,hyperindex,breaklinks]{hyperref}
\hypersetup{
           breaklinks=true,   
           colorlinks=true,   
           pdfusetitle=true,  
        }

\usepackage{orcidlink}
\title[Level-raising of even representations of tetrahedral type]{Level-raising of even representations of tetrahedral type and equidistribution of lines in the projective plane}
\date{\today}
\usepackage[foot]{amsaddr}
\author{Peter Vang Uttenthal \orcidlink{0009-0001-0878-8213}}
\email{petervang@math.au.dk}
\address{Department of Mathematics, Aarhus University, Ny Munkegade 118, 1530-421, DK-8000
Aarhus C, Denmark}
\newcommand{\Gal}{\operatorname{Gal}}

\newcommand{\Q}{\mathbb{Q}}
\newcommand{\Z}{\mathbb{Z}}
\newcommand{\F}{\mathbb{F}}

\newcommand{\Ad}{\operatorname{Ad}^0(\overline{\rho})}
\newcommand{\Adl}{\operatorname{Ad}^0(\overline{\rho}^{(\ell)})}

\usepackage{amsthm,amsmath, mathrsfs, mathtools}
\usepackage{amsfonts}
\usepackage{amssymb}
\usepackage{fancyhdr}
\usepackage{IEEEtrantools}
\usepackage{tikz-cd}
\usepackage[english]{babel}
\usepackage[utf8]{inputenc}

\usepackage{csquotes}

\newtheorem{theorem}{Theorem}
\newtheorem{lemma}[theorem]{Lemma}
\newtheorem{definition}[theorem]{Definition}
\newtheorem{proposition}[theorem]{Proposition}
\newtheorem{corollary}[theorem]{Corollary}
\newtheorem{remark}[theorem]{Remark}
\newtheorem{conjecture}[theorem]{Conjecture}
\newtheorem{question}[theorem]{Question}

\usepackage[hyphens,spaces,obeyspaces]{url}
\usepackage[colorlinks,allcolors=blue,hyperindex,breaklinks]{hyperref}

\hypersetup{
           breaklinks=true,   
           colorlinks=true,   
           pdfusetitle=true,  
        }
\usepackage{tabularx}
\usepackage{subcaption}
\begin{document}

\maketitle

\begin{figure}[h]
    \centering
    \includegraphics[scale=0.5]{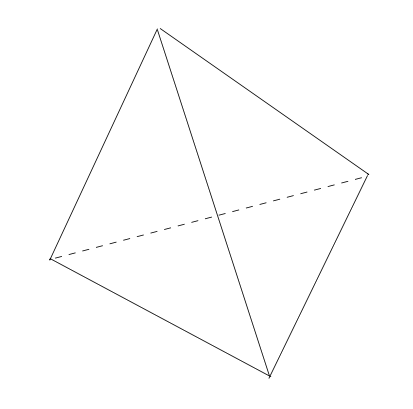}
    \caption{Tetrahedron with symmetry  group $A_4$ 
 (Langlands \cite[p. 16]{langlands})}
\end{figure}

\begin{abstract}
The distribution of primes raising the level of even Galois representations of tetrahedral type is studied.  Data are presented  on primes $v\leq 10^8$ raising the level of $3$-adic even representations of various conductors.
Based on the data, a conjecture is formulated concerning the distribution of certain lines in the plane. 
By an application of Wiles' formula, the conjecture is shown to imply that the density of primes raising the level of a $p$-adic even representation is $$\frac{p-1}{p},$$
in agreement with the density of $2/3$ for $p=3$ observed in the data.  
\end{abstract} \newpage
\tableofcontents
\listoftables

\section{Introduction} 
An odd, complex, 2-dimensional Galois representation is said to have level $N$ if it arises from a modular form of weight 2 and level $N$. In this case, the primes dividing the integer $N$ are exactly the primes at which the Galois representation is ramified.
During the work that lead to a proof of the Shimura-Taniyama-Weil conjecture by Wiles, Taylor and others \cite{Wiles} \cite{Taylor-Wiles}, where geometric Galois representations were under study, 
the need arose for a theory capable of adding new primes to the level. Ribet was the first person to develop such a theory:
By taking advantage of the geometric nature of the representations, he successfully brought in techniques from algebraic geometry
to achieve his theory of level-raising in \cite{Ribet}. 
Unfortunately, even representations are not known to come from geometry.
Nevertheless, in the philosophy of R.  P. Langlands,  complex 2-dimensional even representations should correspond to automorphic representations of $\operatorname{GL}(2)$ over $\Q$ attached to classical Maass wave forms, so a theory of level-raising for even representations is expected to be attainable. The first steps towards such a theory was taken in \cite{even2}, and the present work is a continuation of this endeavour. After identifying under what conditions a prime can be added to the level, it is natural to study how frequently such primes occur. 
Indeed, the distribution of primes subject to various conditions of interest is a classical theme in number theory. This paper presents progress on the distribution of primes raising the level of a certain family of even representations, both from a computational and theoretical point of view. 

In \cite{classification}, it was shown that all $3$-adic even representations of tetrahedral or octahedral type with image contaning $\operatorname{SL}(2,\Z_3)$ and of prime conductor falls into an explicit family $\rho^{(\ell)}$ indexed by primes $\ell$.
In this paper, we will generalize the theory of level raising in \cite{even2} to this family. Then, we will use the three smallest members, $\rho^{(163)}$, $\rho^{(277)}$, and $\rho^{(349)}$, to compute data on the distribution of primes raising their respective levels. 

In \cite{even2}, data on primes $v$ raising the level of $\rho^{(349)}$ was presented for $v$ up to 1000. In this paper, we expand this dataset as follows: 
For the $3$-adic even representations of conductors $\ell = 163,  277$ and $349$, 
we compute all primes $v$ raising their level for $v$ up to $10^8$.
We are interested in how frequently primes can be added to the level;
numerical evidence is presented suggesting that the density of primes $v$ raising the level is $2/3$, independently of the conductor $\ell$. 

In \cite{even2}, 
a Chebotarev condition was designed to address the concern that when a prime $v$ raised the level, the new even representation was not automatically ramified at $v$. In this paper, we show that this concern can be disregarded from a statistical point of view, since ramification at the new prime $v$ almost always happens: the density of primes at which level-raising is not genuine is a thin set of density zero. 

In order to explain how the observed density of primes raising the level arises (the $2/3$-statistics in our data), 
we show that it is a special case of a general theorem:
For any even $p$-adic representation
$$
\rho: \Gal(\overline{\Q} / \Q) \longrightarrow \operatorname{SL}(2, \Z_p),
$$
the density of primes $v$ raising the level of $\rho$ is governed by the probability distribution over primes of certain ramified lines in the plane. 
If this distribution is uniform, then the density of primes raising the level is 
$$\frac{p-1}{p}.$$
We prove the density theorem for general $p$-adic representations by an application of Wiles' formula \cite[Prop. 1.6]{Wiles}. 
When $p=3,$ this density agrees with the 2/3-statistics observed in the data. 
Finally, we formulate a conjecture, jointly with Ramakrishna, that the probability distribution over primes of lines in the plane is indeed uniform. This problem has been open since at least 1998. 

\subsection{Acknowledgments}
I am grateful to Andres Fernandez Herrero, Ravi Ramakrishna, Paul Nelson, Andrew Sutherland, and Mark Watkins for their support. 
This work was supported by research grant VIL54509 from Villum Fonden.

\section{Level-raising at primes $v$}

For a prime $\ell$, let $\Q(\zeta_\ell)$ denote the cyclotomic field 
obtained by adjoining a primitive $\ell$th root of unity $\zeta_\ell$ to $\Q$. 
In \cite{auto}, an explicit family of even residual representations 
$$\overline{\rho}^{(\ell)}: \Gal(\overline{\Q}/\Q) \to \operatorname{SL}(2,\F_3)$$ 
was constructed for primes numbers $\ell$ with the properties that (1) $\ell \equiv 1 \bmod 3$ and (2) the cubic subfield 
$L$ of $\Q(\zeta_\ell)$ has class number $h_L$ divisible by 2. 
Let 
$K=\Q(\overline{\rho}^{(\ell)})$ be the 
 field fixed by the kernel of the  projective representation $\pi \circ \overline{\rho}^{(\ell)}$ obtained by composing $\overline{\rho}^{(\ell)}$ with the natural homomorphism
$$
\pi: \operatorname{SL}(2,\F_3) \to 
\operatorname{PSL}(2,\F_3).
$$
The field $K$  
is totally real with discriminant $\ell^2$ and of tetrahedral type: $\Gal(K/\Q)$ is isomorphic to the symmetry group $A_4$ of a tetrahedron. 

In \cite{classification}, the subset of the family of $\overline{\rho}^{(\ell)}$
that admit surjective lifts 
$$
\rho^{(\ell)}:  \Gal(\overline{\Q}/\Q) \to \operatorname{SL}(2,\Z_3)$$ 
was identified. Here, each $\rho^{(\ell)}$ is ramified only at $\ell$ and at $3$; the primes $S= \{3,\ell \}$ is referred to as the level of $\rho^{(\ell)}$. In this section, we will identify a class of primes $v$ for which there exists an even, surjective representation
$$
\rho^{(\ell, v)}: \Gal(\overline{\Q}/\Q) \to \operatorname{SL}(2,\Z_3)
$$
of level $\{3, \ell, v\}$ such that 
$$
\rho^{(\ell, v)} \equiv \overline{\rho}^{(\ell)} \bmod 3.
$$
If we can find a prime $v$ for which such a $\rho^{(\ell,v)}$ exists, then we say that $v$ raises the level of $\rho^{(\ell)}$ (or that $v$ can be added to the level of $\rho^{(\ell)}$). We will repeat the definition of level-raising below in Definition \ref{deflev}.

We will study level-raising of $\rho^{(\ell)}$ by means of Selmer groups in Galois cohomology, which we will now introduce. For now, 
let $$\overline{\rho}: \Gal(\overline{\Q}/\Q) \to \operatorname{GL}(2,\F_p)$$
be any residual representation over the finite field $\F_p$ with $p$ elements, and let $\Ad$ be the Galois module obtained by
composing $\overline{\rho}$ with the adjoint representation 
of $\operatorname{GL}(2,\F_p)$ on the Lie algebra
$\mathfrak{sl}(2,\F_p)$ of traceless $2\times 2$ matrices over $\F_p$.  
For a place $q$ in $\Q$, a subgroup 
$\mathcal{N}_q$ of the local cohomology group 
$H^1(\Gal(\overline{\Q}_q/\Q_q), \Ad)$
is called a local Selmer condition at $q$.
Let $S$ be a set of places in  $\Q$ containing $p$ and the infinite place $\infty$. A (global) Selmer condtion 
$\mathcal{N}=(\mathcal{N}_q)_{q\in S}$ is defined as a collection of local Selmer conditions at $q\in S$. 
Let $\Q_S$ be the maximal extension of $\Q$ unramified outside $S$. The Selmer group attached to a Selmer condition $\mathcal{N}$ is, by definition, the subgroup of global cohomology classes  
$$f \in H^1(\Gal(\Q_S/\Q),\Ad )$$ with the property that the restriction 
$$
f|_{\Gal(\overline{\Q}_q/\Q_q)} \in \mathcal{N}_q
$$
for all $q\in S$. The Selmer group is denoted by 
$H^1_\mathcal{N}(\Gal(\Q_S/\Q), \Ad).$ 
Similarly, let $\Ad^* =\operatorname{Hom}(\Ad, \mu_p)$, where $\mu_p$ is the $p$th roots of unity. 
Recall that in Galois cohomology, there is a canonical, nondegenerate local pairing on 
$$
H^1(\Gal(\overline{\Q}_q/\Q_q), \Ad)\times H^1(\Gal(\overline{\Q}_q/\Q_q), \Ad^*), 
$$
so we may let $\mathcal{N}_q^\perp$ denote the annihilator in $H^1(\Gal(\overline{\Q}_q/\Q_q)$ with respect to this local pairing. 
The dual Selmer group attached to the dual Selmer condition 
$\mathcal{N}^\perp = (\mathcal{N}_q^\perp)_{q\in S}$
is defined as the subgroup of global cohomology classes 
$$\varphi \in H^1(\Gal(\Q_S/\Q),\Ad^* )$$ 
such that 
$$
\varphi|_{\Gal(\overline{\Q}_q/\Q_q)} \in \mathcal{N}_q^\perp
$$
for all $q\in S$. The dual Selmer group is denoted
$$H^1_{\mathcal{N}^\perp}(\Gal(\Q_S/\Q), \Ad).$$
Having defined these notions, we can state the key idea 
in the Galois cohomological approach to the deformation theory of Galois representations: Identity a Selmer condition $\mathcal{N}$ such that the Selmer group and the dual Selmer group attain the same dimension over $\F_p$.
In this case, we will refer to the pair $(\overline{\rho}, \mathcal{N})$ 
consisting of a residual representation $\overline{\rho}$ unramified outside a finite set of places $S$ and the associated Selmer condition $\mathcal{N}=(\mathcal{N}_q)_{q\in S}$ as a \emph{balanced global setting}, and we will refer to the common dimension, $n$, over $\F_p$ as the \emph{rank} 
of the balanced global setting:
$$
n=
H^1_{(\mathcal{N}_q)_{q\in S}}(\Gal(\Q_{S}/\Q), \Ad)=H^1_{(\mathcal{N}_q^\perp)_{q\in S}}(\Gal(\Q_{S}/\Q), \Ad^*).
$$

To a Selmer condition $\mathcal{N}$, we attach a global deformation ring denoted $R_\mathcal{N}$ whose tangent space is precisely the Selmer group $H^1_\mathcal{N}(\Gal(\Q_S/\Q),\Ad)$. 
For each $q \in S$, suppose we have identified a class $\mathcal{C}_q$ of local deformations 
$\rho$ of $\overline{\rho}|_{\Gal(\overline{\Q}_q/\Q_q)}$ such  that $\rho$ factors through a smooth quotient 
$\phi_q: R_q \to \Z_p[[T_1,\ldots,T_{n_q}]]$ of the local deformation ring $R_q$ at $q$. 
To such a class $\mathcal{C}_q$ of deformations, we choose the Selmer condition $\mathcal{N}_q$ to be the subgroup of $H^1(\Gal(\overline{\Q}_q/\Q_q),\Ad)$ 
obtained from dualizing the map on tangent spaces induces by $\phi_q$.  
Then we define $R_\mathcal{N}$ as the global deformation ring parameterizing deformations $\rho$ of $\overline{\rho}$ with the property that $\rho|_{\Gal(\overline{\Q}_q/\Q_q)} \in \mathcal{C}_q$ for all $q\in S$.

In \cite{classification}, the set $S$ was taken to be $\{3,\ell, \infty\},$ and a Selmer condition was identified for which the global setting was balanced of rank zero (that is, both the Selmer and dual Selmer groups were trivial), and where 
$$R_\mathcal{N} =R_{(\mathcal{N}_q)_{q\in S}} \simeq \Z_3.$$

In this paper, we will identify primes $v$ and a local Selmer condition $\mathcal{N}_v$ such that the global setting remains balanced after allowing ramification at $v$:
$$
H^1_{(\mathcal{N}_q)_{q\in S\cup \{ v \}}}(\Gal(\Q_{S\cup \{ v \} }/\Q), \Ad)=H^1_{(\mathcal{N}_q^\perp)_{q\in S\cup \{ v \}}}(\Gal(\Q_{S\cup \{v \} }/\Q), \Ad^*),
$$
and such that the rank remains zero. In this case, 
$$R_{(\mathcal{N}_q)_{q\in S\cup \{ v \}}} \simeq \Z_3,$$
and therefore the global deformation $\rho^{(\ell,v)}$
attached to the ring $R_{(\mathcal{N}_q)_{q\in S\cup \{ v \}}}$ is the desired even Galois representation raising the level of $\rho^{(\ell)}$.
If the global setting at $S$ is balanced of rank zero, then a sufficient condition for the global setting at $S\cup \{ v\}$ to remain balanced of rank zero is:
$$
\dim \mathcal{N}_v = H^0(\Gal(\overline{\Q}_v/\Q_v), \Ad)
$$
This is an immediate consequence of Wiles' formula \cite[Proposition 1.6]{Wiles}. 

\begin{definition}[Auxiliary primes] \label{aux}
Let $\rho^{(\ell)}$ be a residual even representations of tetrahedral type and level $\ell$. 
Let $C^{(3)}$ be the Chebotarev set of primes $v$ such that 
$v \equiv 1 \mod 3$ and such that the Frobenius automorphisms above $v$ 
in $\Q(\overline{\rho}^{(\ell)})/\Q$ has order 3.
We will call the primes $v$ in $C^{(3)}$ for auxiliary primes
(later, we will define Chebotarev sets $C^{(p)}$ for all $p\geqslant 3.$)
\end{definition}
The class $C^{(3)}$ of auxiliary primes is a class which larger than the set of primes that we eventually will be able to add to the level of $\rho^{(\ell)}$. It will be convenient to work with the class $C^{(3)}$ because all primes we can add to the level will necessarily lie in $C^{(3)}$. Below, we will be taking primes $v\in C^{(3)}$, and then we will identify what condition(s) we need the auxiliary prime $v$ to satisfy in order to raise the level of $\rho^{(\ell)}.$

Let $v \in C^{(3)}$. For ease of notation, let 
$G_v:= \Gal(\overline{\Q}_v/\Q_v)$. 
Recall that the maximal tame quotient of $G_v$ is topologically generated by
$\sigma_v$ (a lift of the Frobenius automorphism at $v$) and inertia $\tau_v$, subject to the relation $\sigma_v \tau_v \sigma_v^{-1}= (\tau_v)^v$. 
Since the standard $efg$-decomposition of $v$ in the number field $K=\Q(\overline{\rho}^{(\ell)})$ 
is $$e=1, f=3, g=4,$$ 
we may fix a basis such that 
$$
\overline{\rho}^{(\ell)}|_{G_v}: 
\sigma_v \mapsto 
\begin{pmatrix}
1 & 1 \\
0 & 1
\end{pmatrix},
\quad 
 \tau_v \mapsto 
\begin{pmatrix}
1 & 0 \\
0 & 1
\end{pmatrix}.
$$
\begin{lemma} \label{localranks3}
The local cohomology groups at $v$ have the following ranks:
$$
\dim H^0(G_v, \operatorname{Ad}^0(\overline{\rho}^{(\ell)})) =
\dim H^2(G_v, \operatorname{Ad}^0(\overline{\rho}^{(\ell)})) = 1,
\quad 
\dim H^1(G_v, \operatorname{Ad}^0(\overline{\rho}^{(\ell)})) = 2.
$$
\end{lemma}
\begin{proof}
Since $H^0(G_v, \operatorname{Ad}^0(\overline{\rho}^{(\ell)}))$ is defined as the subspace of $\operatorname{Ad}^0(\overline{\rho}^{(\ell)})$
invariant under the Galois action, it is immediate that $\dim H^0(G_v, \operatorname{Ad}^0(\overline{\rho})) =1$.
Local duality implies that 
$$
\dim H^2(G_v, \Ad) =
\dim H^0(G_v, \Ad^*) = 1.
$$
Finally, it follows from the local Euler characterstic that $\dim H^1(G_v, Ad^0(\overline{\rho}))=2$.
\end{proof}
Let $\mathcal{C}_v$ be the set of deformations of 
$\overline{\rho}|_{G_v}$ to $\Z/3^n \Z$ for integers $n\geqslant 1$ of the form
$$
\sigma_v \mapsto 
\begin{pmatrix}
\sqrt{v} & 1 + 3x \\
0 & \sqrt{v}^{-1}
\end{pmatrix},
\quad 
 \tau_v \mapsto 
\begin{pmatrix}
1 & 3y \\
0 & 1
\end{pmatrix}.
$$
Note that $v\equiv 1 \mod 3$ implies $\sqrt{v} \in 1+3\Z_3.$

\begin{lemma} \label{v}
Let $v\in C^{(3)}$. 
Let  $\mathcal{N}_v$ be the subspace of $ H^1(G_v, \Adl)$ that preserves $\mathcal{C}_v$.
The local Selmer condition $\mathcal{N}_v$ has  
$$
\dim \mathcal{N}_v = \dim H^0(G_v, \operatorname{Ad}^0(\overline{\rho}^{(\ell)})) .
$$
In fact, $\mathcal{N}_v$ is spanned by the local cohomology class
$$
r_v: 
\sigma_v \mapsto 
\begin{pmatrix}
0 & 0 \\
0 & 0
\end{pmatrix},
\quad 
 \tau_v \mapsto 
\begin{pmatrix}
0 & 1 \\
0 & 0
\end{pmatrix}.
$$
\end{lemma}

\begin{definition} \label{split}
    Let $v\in C^{(3)}$ and let $f_v \in H^1(G_v, \operatorname{Ad}^0(\overline{\rho}^{(\ell)})).$
Then $f_v$ gives rise to a deformation of $\overline{\rho}|_{G_v}$ to the dual numbers $\F_3[\varepsilon]$ whose kernel fixes an extension $E_{f_v}/\Q_v$, and  
   by local class field theory, $E_{f_v}/\Q_v$ has degree 9.
 We say that $f_v$ is \emph{split} if the exact sequence
    $$
    1 \to \Gal(E_{f_v}/\Q_v(\overline{\rho}|_{G_v})) \to \Gal(E_{f_v} /\Q_v) \to \Gal( \Q_v(\overline{\rho}|_{G_v}) /\Q_v) \to 1
    $$
splits, and otherwise we say that $f_v$ is \emph{nonsplit}.
\end{definition}

\begin{lemma} \label{N_v}
Let $v \in C^{(3)}$. 
The local cohomology group $H^1(G_v, \Adl)$ is spanned by a split and a nonsplit cohomology class, and the local Selmer condition $\mathcal{N}_v$ is the subspace 
spanned by a split class. 
\end{lemma}
\begin{proof}
    This is immediate from Lemma \ref{v} and Definition \ref{split}.
\end{proof}

\begin{lemma} \label{remainbalanced3}
Suppose the global setting is balanced at $S$.
If we allow ramification at $v\in C^{(3)}$, then the global setting remains balanced in the sense that
    $$
\dim H^1_{\mathcal{N}} (G_{S\cup\{v\}}, \operatorname{Ad}^0(\overline{\rho}^{(\ell)}))
=\dim H^1_{\mathcal{N}^\perp} (G_{S\cup\{v\}}, \operatorname{Ad}^0(\overline{\rho}^{(\ell)})^*).$$
\end{lemma}
\begin{proof}
    This follows from Lemma \ref{v} and Wiles' formula \cite[Prop 1.6]{Wiles}. 
\end{proof}

\begin{lemma}
The density of primes $v$ such that $v\in C^{(3)}$ is 1/3.
\end{lemma}
\begin{proof}
There are four conjugacy classes in $A_4$:
\begin{IEEEeqnarray*}{rCl}
C_1 & := &\{ 1\}, \\
C_2 & := &\{ (13)(23), (14)(23)\}, \\
C_3 & := &\{ (123), (243), (134), (142)\}, \\
C_4 & := & \{ (132), (234), (143), (124)\}.
\end{IEEEeqnarray*}
A prime $v$ has inertial degree $f(v,K/\Q)= 3$ if and only if
the conjucagy class of Frobenius automorphisms above $v$ is either $C_3$ or $C_4$.
Let $(v,K/\Q)$ denote the conjugacy class of Frobenius automorphisms at $v$ in $K/\Q$.
By Chebotarev's theorem, the density of primes $v$ such 
that $v\equiv 1\mod 3$ and such that $(v, K/\Q) = C_3$ or $(v, K/\Q) = C_4$ 
is 
$$
\frac{1}{|(\Z/3\Z)^\times|}
\bigg( \frac{|C_3|}{|A_4|} +\frac{|C_4|}{|A_4|} \bigg)  
=
\frac{1}{2}\bigg( \frac{4}{12}+\frac{4}{12} \bigg) = 1/3.
$$
\end{proof}

For a number field $E/\Q$, a prime $p$, and a set of places $S$, let $E^{(p)}_S$ denote the maximal $p$-elementary abelian extension of $E$ unramified away from $S$. The field $E^{(p)}_S$ is also called the $p$-Frattini extension of $E$ unramified outside $S$.

\begin{lemma} \label{T_v}
    Let $v\in C^{(3)}$. Let $T_v$ be the primes in $K$ above $v$ and $3$.
    Let $K^{(3)}_{T_v}$ be the 3-Frattini extension of $K$ unramified outside $T_v$.
    Then
    $$
\Gal(K^{(3)}_{T_v} / K) \simeq \Adl \oplus \F_3 \oplus \F_3 
    $$
as $\F_3[\Gal(K/\Q)]$ modules. 
\end{lemma}

\begin{proof}
    This follows directly from the formula for the 
    rank of $3$-Frattini extensions 
    from global class field theory \cite[Chapter 8]{koch}. 
\end{proof}
Recall that $\Q_S$ denotes the maximal extension of $\Q$ unramified outside $S$. Let $G_S =\Gal(\Q_S/\Q)$.
\begin{lemma} \label{sv}
    Let $v\in C^{(3)}$ and 
let $T_v$ be the primes in $K$ above $v$ and $3$.
 There is a unique number field $K^{(v)}/K$ contained in $K^{(3)}_{T_v}$ such that 
    $
\Gal(K^{(v)}/K) \simeq \Ad
    $
    as $\F_3[\Gal(K/\Q)]$ modules.
    The number field $K^{(v)}$ corresponds to a cohomology class $s^{(v)} \in H^1(G_S, \Ad)$ which is unramified at $\ell$.
    The decomposition group at $v$ in $\Gal(K^{(v)}/K)$ has order 9, and $K^{(v)}/K$ is ramified of degree 3 at primes in $K$ above $v.$
\end{lemma}
\begin{proof}
   The existence of $K^{(v)}$ follows by Lemma \ref{T_v}. The cohomology class $s^{(v)}$ is unramified at $\ell$ because $T_v$ does not contain the primes in $K$ above $\ell$.
\end{proof}

\begin{lemma} \label{nonsplit}
Let 
$\ell \equiv 1 \mod 3$ with 4 dividing the class number of the cubic subfield of $\Q(\zeta_\ell),$
and let $K=\Q(\overline{\rho})$ be the corresponding
$A_4$ extension.  
Let $S=\{\infty, 3, \ell \}.$
If the primes above $v$ in $K$ have Frobenius automorphisms in $K^{(3)}_{T_v}$ of order 3, then $s^{(v)}|_{G_v}$ is nonsplit, 
    $s^{(v)}|_{G_v} \notin \mathcal{N}_v,$ and the global setting is balanced of rank zero after allowing ramification at $v,$ i.e.  
       $$
H^1_{\mathcal{N}} (G_{S\cup\{v\}}, \operatorname{Ad}^0(\overline{\rho}^{(\ell)}))
=
H^1_{\mathcal{N}^\perp} (G_{S\cup\{v\}}, \operatorname{Ad}^0(\overline{\rho}^{(\ell)})^*)
= 0.$$
\end{lemma}
\begin{proof}
The proof of 
By Lemma \ref{N_v}, $s^{(v)}|_{G_v}$ is nonsplit if and only if $s^{(v)}|_{G_v} \notin \mathcal{N}_v$, and in this case, 
    $s^{(v)} \notin H^1_{\mathcal{N}} (G_{S\cup\{v\}}, \operatorname{Ad}^0(\overline{\rho})).$
    The global setting is balanced of rank zero at $S$, and by Lemma \ref{remainbalanced3} the global setting remains balanced of rank zero after allowing ramification at $v$.
\end{proof}

\begin{definition}[Level raising at $v$] \label{deflev}
    Suppose the global setting is balanced of rank zero at $S$. Then there is a unique even 3-adic representation
    $$
\rho: G_S \longrightarrow \operatorname{SL}(2,\Z_3)
    $$
    such that $\rho|_{G_q} \in \mathcal{C}_q$ for all
    $q\in S.$
    If the global setting remains balanced of rank zero after allowing ramification at $v$, we say that level raising at $v$ is possible for $\rho,$ or that $v$ raises the level of $\rho$ in the sense of Galois cohomology. 
    In this case, there is a unique even 3-adic representation 
    $$ 
   \rho^{(v)}:
    G_{S\cup \{v\}} \longrightarrow \operatorname{SL}(2,\Z_3).
    $$
    such that $\rho^{(v)} \equiv \overline{\rho} \mod 3$ and
     $\rho^{(v)}|_{G_q} \in \mathcal{C}_q$ 
     for all $q\in S\cup \{v \}.$
\end{definition}

\subsection{Ramification at the auxiliary prime}
\begin{proposition} \label{ram}
Let $\rho^{(\ell)}$ be the surjective even representation onto $\operatorname{SL}(2,\Z_3)$
lifting $\overline{\rho}^{(\ell)}$.
Let $v\in C^{(3)}$ be an auxiliary prime raising the level of $\rho^{(\ell)}$ according to Definition \ref{deflev}.
For ease of notation, let $\rho^{(\ell, v)}$ denote the new $3$-adic even representation whose level has been raised to $S \cup \{ v\}.$
If $\rho^{(\ell, v)}$ is unramified at $v$, then 
$\rho^{(\ell, v)}$ factors through $G_S,$ and 
$$ \rho^{(\ell,v)} = \rho^{(\ell)}.$$
 Among the set of primes $v$ raising the level of $\rho^{(\ell)}$,
 the subset of primes $v$
 such that $\rho^{(\ell, v)}$ is unramified at $v$ has   density zero. 
\end{proposition}

\begin{proof}
Since $\rho^{(\ell,v)}|_{G_v} \in \mathcal{C}_v,$
there are $x$ and $y$ in $\Z_3$ such that 
$$
\rho^{(\ell,v)}|_{G_v}: 
\sigma_v \mapsto 
\begin{pmatrix}
\sqrt{v} & 1 + 3x \\
0 & \sqrt{v}^{-1}
\end{pmatrix},
\quad 
 \tau_v \mapsto 
\begin{pmatrix}
1 & 3y \\
0 & 1
\end{pmatrix}.
$$
We see that $\rho^{(\ell, v)}|_{G_v} $ is unramified at $v$ if and only if $y=0,$
which happens with probability zero. 
More precisely, let $\rho_n = \rho^{(\ell, v)} \mod 3^n$ for $n\geq 1$: 
    $$
\rho_n: G_S \longrightarrow \operatorname{SL}(2, \Z /3^n \Z).
    $$
    Since $\rho_n|_{G_v} \in \mathcal{C}_v$, 
$$
\rho_n|_{G_v}: 
\sigma_v \mapsto 
\begin{pmatrix}
\sqrt{v} & 1 + 3x_n \\
0 & \sqrt{v}^{-1}
\end{pmatrix},
\quad 
 \tau_v \mapsto 
\begin{pmatrix}
1 & 3y_n \\
0 & 1
\end{pmatrix}
$$
for some $x_n, y_n \in \Z/3^n\Z,$ and $\rho_n$ is unramified at $v$ if and only if $y_n=0$ 
in $\Z/3^n\Z,$
which happens with probability $1/3^n.$
Then $\rho^{(\ell,v)}$ is unramified at $v$ if and only if $y_n = 0$ for all $n$, which happens with probability zero. 
\end{proof}

\subsection{Descent}
 
Let $v \in C^{(3)}$. Let $F$ be the subfield of $K$ fixed by a subgroup 
of order 3 in $A_4$.
   If $\mathcal{O}_F$ is the ring of integers in $F$, then
   $$
    v \mathcal{O}_F = \mathfrak{v}_1\mathfrak{v}_2, \quad 3\mathcal{O}_F= 3_1 3_2, \quad \ell \mathcal{O}_F = \ell_1 \ell_2^3.
   $$
   where the norm of $\mathfrak{v}_1$ is $v^3$, the norm of $\mathfrak{v}_2$ is $v$, the norm of $3_1$ is $27$ and the norm of $3_2$ is 3. 

\begin{lemma}[Sufficient condition for level raising at $v$]
Let $C^{(\lambda)}$ be the subset of $v\in C^{(3)}$ such that
the prime ideal $\mathfrak{v}_1$ in $F$ of norm $v^3$ remains prime in the 3-Frattini extension 
$F_{\{  3_1, \mathfrak{v}_2 \} }^{(3)}/F$ unramified outside  $3_1$ and $\mathfrak{v}_2,$
i.e. 
 $$
C^{(\lambda)}=\{ v\in C^{(3)}: f( \mathfrak{v}_1, F_{\{  3_1, \mathfrak{v}_2 \} }^{(3)}/F)= 3 \}. 
$$ 
If $v\in C^{(\lambda)}$, then the global setting remains balanced of rank zero after allowing ramification at $v,$ i.e. 
  $$
\dim H^1_{\mathcal{N}} (G_{S\cup\{v\}}, \operatorname{Ad}^0(\overline{\rho}))
=\dim H^1_{\mathcal{N}^\perp} (G_{S\cup\{v\}}, \operatorname{Ad}^0(\overline{\rho})^*)=0.$$
\end{lemma}

\begin{proof}
   If $v\in C^{(\lambda)},$ then the cohomology class $s^{(v)}$ defined in Lemma \ref{sv} is nonsplit, and by Lemma \ref{nonsplit}, this implies that the global setting remains balanced of rank zero after allowing ramification at $v$. 
\end{proof}

A set of primes $C$ is said to be a \emph{Chebotarev set} 
if $C$ is the set of primes satisfying a finite list of prescribed splitting conditions in a fixed number field $E$ independent of any of the primes in $C$. 

Note that the abelian extension $F_{\{  3_1, \mathfrak{v}_2 \} }^{(3)}/F$ depends on the prime $v$, and therefore $C^{(\lambda)}$ is most likely \emph{not} a Chebotarev set, and the density of $C^{(\lambda)}$ in $C^{(3)}$ cannot be found by Chebotarev's theorem. 
It is open whether $C^{(\lambda)}$ is an infinite set. 

\begin{question} \label{2/3}
Let 
$S=\{  \infty, 3, \ell \}$ and consider a surjective, even representation
  $$
  \rho^{(\ell)}: 
\Gal(\Q_S/\Q)
  \to \operatorname{SL(2,\Z_3})
$$
unramified away from $3$ and $\ell$.
Is the density of $C^{(\lambda)}$ in $C^{(3)}$  
2/3? The data in Tables \ref{table:ell=163}, \ref{table:ell=277} and \ref{table:ell=349} suggest so. 
\end{question}

\begin{corollary} Suppose Question \ref{2/3} has an affirmative answer.  
Let $R_{(\mathcal{N}_q)_{q\in S}}$ be the global even deformation ring parametrizing deformations $\rho$ of
$
\overline{\rho}
$ such that $\rho|_{G_q} \in \mathcal{C}_q$ for all $q\in S.$
If the global setting is balanced of rank zero at $S$, then 
$
R_{(\mathcal{N}_q)_{q\in S}} \simeq \Z_3.
$
The density of auxiliary primes $v \in C^{(3)}$ 
such that 
$$
R_{\mathcal{N}_{q \in S\cup \{v\} }} \simeq \Z_3
$$
is 2/3. 
\end{corollary}

\begin{lemma}[Sufficient condition for ramification at $v$]
    Let $C^{(\tau)}$ be the set of $v\in C^{(3)}$ such that 
the prime ideal $\mathfrak{v}_1$ in $F$ of norm $v^3$ splits completely in the 3-Frattini extension 
        $F_{\{  3_1, \ell_2 \} }^{(3)}/F$ unramified outside $3_1$ and $\ell_2$, i.e.
    $$
C^{(\tau)}=\{ v\in C^{(3)}: f( \mathfrak{v}_1, F_{\{  3_1, \ell_2 \} }^{(3)}/F)= 1  \}. 
$$
If $v\in C^{(\tau)}\cap C^{(\lambda)}$, then $\rho^{(\ell,v)}$ is ramified at $v$ mod 9. 
The density of $C^{(\tau)}$ in $C^{(3)}$ is
$$
\lim_{n\to \infty} \frac{|C^{(\tau)}\cap \{ v \leqslant n\}|}{|C^{(3)}\cap \{ v\leqslant n\}|} =1/3.
$$
\end{lemma}

\begin{proof}
It is easy to verify that if $v \in C^{(\tau)}\cap C^{(\lambda)},$ then 
$\rho^{(\ell,v)}$ is ramified at $v$ mod 9
(see also \cite{even2}). 
By global class field theory, 
$$
\Gal(F_{\{  3_1, \ell_2 \} }^{(3)}/F) \simeq \Z/3\Z.
$$
By Chebotarev's theorem, the density of primes $\mathfrak{v}$ in $F$ such that 
$f( \mathfrak{v} , F_{\{  3_1, \ell_2 \} }^{(3)}/F)= 1$ is equal to 1/3.
\end{proof}

\begin{question}[Independence]
Are the conditions defining $C^{(\lambda)}$ and $C^{(\tau)}$ independent? 
If so, the density 
of $C^{(\lambda)} \cap C^{(\tau)}$ in $C^{(3)}$  is $(1/3)(2/3)=2/9=0.2222\overline{2}.$ 
This is suggested by the data in Tables \ref{table:ell=163}, \ref{table:ell=277} and \ref{table:ell=349}.
More generally, is the level-raising condition
defining $C^{(\lambda)}$ independent of any Chebotarev set? Numerical computations suggest that the level-raising condition
defining $C^{(\lambda)}$ is independent of all Dirichlet conditions modulo $N$ for all $N \leq 20.$
\end{question}

\section{Data}
\begin{definition}
For any set of primes $C^{(\cdot)},$ let $C^{(\cdot)}_n = C^{(\cdot)}\cap \{v\leqslant n \}. $
\end{definition}
In this section, all data are for $p=3$. 
\begin{table}[h!]
\begin{center}
\begin{tabular}{lllllr}  
 $n$ &$|C^{(p)}_n|$ & $\frac{|C^{(\lambda)}_n|}{|C^{(p)}_n|}$ & 
 $\frac{|C^{(\tau)}_n|}{|C^{(p)}_n|}$ & 
 $\frac{|C^{(\lambda)}_n|}{|C^{(p)}_n|} \frac{|C^{(\tau)}_n|}{|C^{(p)}_n|} $ &
 $ \frac{ |C^{(\lambda)}_n \cap C^{(\tau)}_n| }{|C^{(p)}_n|}$ \\ [1ex]
 \hline 
  1,000 & 55 & 0.69091 & 0.27273 & 0.18843 & 0.16364 \\ 
  5,000 & 221 & 0.70588 & 0.31222 & 0.22039 & 0.19910 \\ 
  50,000 & 1709 & 0.67934 & 0.33119 & 0.22499 & 0.21650 \\
  100,000 & 3169 & 0.67214 & 0.33197 & 0.22313 & 0.22121 \\
  400,000 & 11260 & 0.67558 & 0.32709 & 0.22097 & 0.22131 \\
  500,000 & 13829 & 0.67366 & 0.32938 & 0.22189 & 0.22012 \\
  1,000,000 & 26116 & 0.66959 & 0.33183 & 0.22219 & 0.22285 \\  
  10,000,000 & 221581 & 0.66765 & 0.33418 & 0.22311 & 0.22313 \\ 
  50,000,000 & 1000659 & 0.66698 & 0.33414 & 0.22286 & 0.22269 \\
  100,000,000 & 1920314 & 0.66622 & 0.33427 & 0.22270 & 0.22268\\ [1ex]
\end{tabular}
\caption{Density of level-raising primes at $\ell=163$}
\label{table:ell=163}
\end{center}
\end{table}

\begin{table}[h!] 
\begin{center}
\begin{tabular}{lllllr}  
 $n$ &$|C^{(p)}_n|$ & $\frac{|C^{(\lambda)}_n|}{|C^{(p)}_n|}$ & 
 $\frac{|C^{(\tau)}_n|}{|C^{(p)}_n|}$ & 
 $\frac{|C^{(\lambda)}_n|}{|C^{(p)}_n|} \frac{|C^{(\tau)}_n|}{|C^{(p)}_n|} $ &
 $ \frac{ |C^{(\lambda)}_n \cap C^{(\tau)}_n| }{|C^{(p)}_n|}$ \\ [1ex]
 \hline 
  1,000 & 53  & 0.67925 & 0.33962 & 0.23069 & 0.26415 \\ 
  5,000 & 220 & 0.65455 & 0.31818 & 0.20826 & 0.21818 \\ 
  50,000 & 1714 & 0.65111 & 0.33722  & 0.21957 & 0.22054 \\
  100,000 & 3195 & 0.65728 & 0.33083 & 0.21745 & 0.21565 \\
  400,000 & 11278 & 0.65934 & 0.33233 & 0.21912 & 0.21812 \\
  500,000 & 13811 & 0.66099 & 0.33068 & 0.21858 & 0.21773  \\
  1,000,000 & 26191 & 0.66347 & 0.33386 & 0.22150 & 0.22183 \\  
  10,000,000 & 221321 & 0.66575 & 0.334649 & 0.22279 & 0.22192 \\ 
  50,000,000 & 1000037 & 0.66658 & 0.33345 & 0.22227 & 0.22188 \\
  100,000,000 & 1920404 & 0.66637 &  0.33346 & 0.22221 & 0.22190 \\ [1ex]
\end{tabular}
\caption{Density of level-raising primes at $\ell=277$}
\label{table:ell=277}
\end{center}
\end{table}

\begin{table}[h!]
\begin{center}
\begin{tabular}{lllllr}  
 $n$ &$|C^{(p)}_n|$& $\frac{|C^{(\lambda)}_n|}{|C^{(p)}_n|}$ & 
 $\frac{|C^{(\tau)}_n|}{|C^{(p)}_n|}$ & 
 $\frac{|C^{(\lambda)}_n|}{|C^{(p)}_n|} \frac{|C^{(\tau)}_n|}{|C^{(p)}_n|} $ &
 $\frac{ |C^{(\lambda)}_n \cap C^{(\tau)}_n| }{|C^{(p)}_n|}$ \\ [1ex]
 \hline 
  1,000 &  50 & 0.66000 & 0.32000 & 0.21120 & 0.20000 \\ 
  5,000 & 227 & 0.64758 & 0.32599 & 0.21110 & 0.20705 \\ 
  50,000 & 1709 & 0.68227 & 0.32768 & 0.22356 & 0.22586\\
  100,000 & 3174 & 0.66635 & 0.32672 & 0.21771 & 0.21834\\
  400,000 & 11186 & 0.66905 & 0.33086 & 0.22136 & 0.22180\\
  500,000 & 13760 & 0.66294 & 0.33256 & 0.22046 & 0.21969\\
  1,000,000 & 26077 & 0.66373 & 0.33263 & 0.22078 & 0.22000\\  
  10,000,000 & 221486 & 0.66547 & 0.33414 & 0.22236 & 0.22210\\ 
  50,000,000 & 1000434 & 0.66606 & 0.33482 & 0.22301 & 0.22293\\
  100,000,000 &  1929529 & 0.66642 & 0.33399 & 0.22258 & 0.22256 \\
  500,000,000 & 8783915 & 0.66672 & 0.33352 & 0.22236 & 0.22238\\
  1,000,000,000 & 16947027 & 0.66659 & 0.33347 & 0.22229 & 0.22227\\  [1ex]
\end{tabular}
\caption{Density of level-raising primes at $\ell=349$}
\label{table:ell=349}
\end{center}
\end{table}

In Table \ref{table:std}, we consider the largest possible sample size $|C^{(p)}_n|$, 
and focus on the variable $|C^{(\lambda)}_n|.$
For different $\ell$, we report the quantity
$$
\frac{(2/3)|C^{(p)}_n|-|C^{(\lambda)}_n|}{\sqrt{|C^{(p)}_n|}}
$$
which measures standard deviations from the expected value of
the variable $|C^{(\lambda)}_n|.$ As one would expect,
the actual values fall within a standard deviation of the expected values. 
\begin{table}[h!]
\begin{center}
\begin{tabular}{lr}  
 $\ell$ & $\frac{(2/3)|C^{(p)}_n|-|C^{(\lambda)}_n|}{\sqrt{|C^{(p)}_n|}}$   \\ [1ex]
 \hline 
  163 & 0.61897  \\ 
  277 & 0.41112 \\ 
  349 & 0.31797 \\  [1ex]
\end{tabular}
\caption{Standard deviations from the mean}
\label{table:std}
\end{center}
\end{table}

\section{Density of level-raising primes for $\operatorname{GL}(2,\Z_p)$}
\subsection{On the rarity of even representations}
\begin{proposition}
Even Galois representations have been explicitly constructed with each of the following images:
\[
\begin{tikzcd}
   & &  \operatorname{SL}(2, \F_2) \\
   & & \operatorname{SL}(2, \F_3)\\ 
\Gal(\overline{\Q}/\Q) 
 \arrow[rr, "\overline{\rho}^{(5)}" description] 
 \arrow[urr, "\overline{\rho}^{(3)}" description] 
 \arrow[uurr, "\overline{\rho}^{(2)}" description] 
 \arrow[drr, "\overline{\rho}^{(7)}" description] 
 \arrow[ddrr, "\overline{\rho}^{(11)}" description] 
 & & \operatorname{SL}(2, \F_5) \\
       & & \operatorname{SL}(2, \F_7) \\
         &   & \operatorname{SL}(2, \F_{11}) \\
\end{tikzcd}
\]
\end{proposition}
\begin{proof}
Note that $\operatorname{SL}(2, \F_2) \simeq S_3$, and totally real $S_3$-extensions are straightforward to construct. 
In this paper, we have constructed several examples of even representations with image $\operatorname{SL}(2, \F_3).$ 
Furthermore, $\operatorname{PSL}(2,\F_5) \simeq A_5$, and it is not hard to construct totally real $A_5$-extensions and then lift these to totally real $\operatorname{SL}(2,\F_5)$-extensions. 
Zeh-Marschke \cite{ZM} showed that $\operatorname{PSL}(2,\F_7)$ occurs as the Galois group of a totally real extension over $\Q$ by writing down the explicit polynomial defining the number field, and then lifting to a totally real $\operatorname{SL}(2,\F_7)$-extension. 
In addition, it is a consequence of a theorem of Mestre \cite{mestre} that there exist infinitely many $\operatorname{SL}(2,\F_7)$-extensions of $\Q.$
In \cite{kluners}, it is shown that the Galois group of the splitting field of the polynomial
\begin{IEEEeqnarray*}{rCl}
    f(x) &=& x^{11}-4x^{10}-25x^9+81x^8+237x^7-562x^6-1010x^5+1574x^4 \\
&& +1805x^3-1586x^2-847x+579
\end{IEEEeqnarray*}
is $\operatorname{PSL}(2, \F_{11})$, and that there is a totally real quadratic extension
of the splitting field of $f$ whose Galois group over $\Q$ is $\operatorname{SL}(2,\F_{11})$.
\end{proof}

\begin{remark}
    It is not known if there exists an explicit, \emph{even} irreducible Galois representation
    $$
\Gal(\overline{\Q}/\Q) \longrightarrow \operatorname{SL}(2,\F_{13}),
    $$
    and similarly for all $p \geqslant 17.$ Equivalently, polynomials defining
    totally real
    number fields $K^{(p)}$ with 
    $$
\Gal(K^{(p)}/\Q) \simeq \operatorname{SL}(2,\F_{p})
    $$
    for $p \geqslant 13$ are not known. 
\end{remark}

\subsection{Selmer conditions at nice primes}

\begin{definition}[Nice primes] \label{nice}
Let $p\geq 5$, let $S$ be a finite set of places of $\Q$ containing $p$ and the archimedian place.
Let $G_S = \Gal(\Q_S/\Q),$ where $\Q_S$ is the maximal extension of $\Q$ unramified outside $S$.
Consider an irreducible representation
$$
\overline{\rho}: G_S \to \operatorname{GL}(2, \F_p)
$$
with image containing  $\operatorname{SL}(2, \F_p).$
A prime $v$ is nice for $\overline{\rho}$ if 
    \begin{enumerate}
        \item $v \not \equiv \pm1 \mod p$
        \item $\overline{\rho}$ is unramified at $v$
        \item $\overline{\rho}(\sigma_v)$ has eigenvalues of ratio $v$, where $\sigma_v$ denotes the Frobenius automorphism at $v$. 
    \end{enumerate}
    Define $C^{(p)}$ as set of nice primes for $\overline{\rho}.$ This is a  Chebotarev set, and we will also call $C^{(p)}$ the set of auxiliary primes.
\end{definition}

\begin{lemma} \label{localranks5} Let $p\geqslant 5$ and let $v$ be a nice prime. 
Suppose $\overline{\rho}$ is irreducible with image containing
$\operatorname{SL}(2,\F_{p}).$
The local cohomology groups at $v$ have the properties that
$$
\dim H^0(G_v, \operatorname{Ad}^0(\overline{\rho})) = 1,
\quad 
\dim H^2( G_v, \operatorname{Ad}^0(\overline{\rho})) = 1,
\quad 
\dim H^1(G_v, \operatorname{Ad}^0(\overline{\rho})) = 2.
$$
\end{lemma}
\begin{proof} See \cite{KLR1} p. 714. 
\end{proof}
Let $p\geqslant 5$ and let $v$ be a nice prime.   
Up to a twist (cf. \cite{KLR1} p. 714),
$$
\overline{\rho}|_{G_v}: 
\sigma_v \mapsto 
\begin{pmatrix}
v & 0 \\
0 & 1
\end{pmatrix},
\quad 
 \tau_v \mapsto 
\begin{pmatrix}
1 & 0 \\
0 & 1
\end{pmatrix}.
$$
Let $\mathcal{C}_v$ be the set of deformations of 
$\overline{\rho}|_{G_v}$ to $\Z/p^n\Z$ for $n\geq 1$ of the form
$$
\pi_v: 
\sigma_v \mapsto 
\begin{pmatrix}
v & 0 \\
0 & 1
\end{pmatrix},
\quad 
 \tau_v \mapsto 
\begin{pmatrix}
1 & py \\
0 & 1
\end{pmatrix}.
$$
\begin{lemma} \label{vnice}
Let $p \geqslant 5$ and let $v$ be a nice prime. 
Suppose $\overline{\rho}$ is irreducible with image containing
$\operatorname{SL}(2,\F_{p}).$
Let  $\mathcal{N}_v$ be the maximal subspace of $ H^1(G_v, \Ad)$ that preserves $\mathcal{C}_v$.
Then 
$$
\dim \mathcal{N}_v = \dim H^0(G_v, \operatorname{Ad}^0(\overline{\rho}) ),
$$
and $\mathcal{N}_v$ is spanned by the local cohomology class
$$
r_v: 
\sigma_v \mapsto 
\begin{pmatrix}
0 & 0 \\
0 & 0
\end{pmatrix},
\quad 
 \tau_v \mapsto 
\begin{pmatrix}
0 & 1 \\
0 & 0
\end{pmatrix}.
$$
\end{lemma}
\begin{lemma} \label{remainbalanced5}
Let $p\geqslant 5.$
    Suppose 
    $$
    \overline{\rho}: G_S \to \operatorname{GL}(2, \F_{p})
    $$
    is irreducible with image containing
$\operatorname{SL}(2,\F_{p})$ and that the global setting is balanced. 
    If we allow ramification at an auxiliary nice prime $v$, then the global setting remains balanced in the sense that
    $$
\dim H^1_{\mathcal{N}} (G_{S\cup\{v\}}, \operatorname{Ad}^0(\overline{\rho}))
=\dim H^1_{\mathcal{N}^\perp} (G_{S\cup\{v\}}, \operatorname{Ad}^0(\overline{\rho})^*).$$
\end{lemma}
\begin{proof}
    This follows from Lemma \ref{vnice} and Wiles' formula. 
\end{proof}

\section{Density $(p-1)/p$ and equidistribution of lines in the projective  plane}

\begin{lemma}  \label{kernel}
Let $p\geqslant 3.$
Suppose the global setting is balanced of rank zero at the minimal level, i.e. 
$$ H^1_{\mathcal{N}} (G_{S}, \Ad)
= H^1_{\mathcal{N}^\perp} (G_{S}, \Ad^*) = 0.$$
Let $p \geqslant 3,$
and let $v \in C^{(p)}$ (cf. Definition \ref{aux} if $p=3$ and Definition \ref{nice} if $p\geq 5$.)  
There is one global cohomology class $h^{(v)}$ in the kernel of the localization map
$$
H^1(G_{S\cup\{v\}}, \Ad) 
\longrightarrow
\bigoplus_{q\in S} H^1(G_q, \Ad)/\mathcal{N}_q,
      $$
and $h^{(v)}$ is ramified at $v$.
\end{lemma}
\begin{proof}
The statement follows from Wiles' formula \cite[Prop. 1.6]{Wiles}. If $h^{(v)}$ was unramified at $v$, then 
    $h^{(v)} \in H^1_\mathcal{N}(G_{S}, \Ad) = 0.$
    Hence $h^{(v)}$ is ramified at $v$.
\end{proof}

\begin{theorem} \label{(p-1)/p}
Let $p \geqslant 3$,
and let $v \in C^{(p)}$ (cf. Definition \ref{aux} if $p=3$ and Definition \ref{nice} if $p\geq 5$.) 
Let $h^{(v)}$ be a generator of the kernel of the homomorphism
$$ 
H^1(G_{S\cup\{v\}}, \Ad) 
\longrightarrow
\bigoplus_{q\in S} H^1(G_q, \Ad)/\mathcal{N}_q.
$$
There are $p$ ramified lines in the $\F_p$-vector space $H^1(G_v, \Ad)$
of which  $\mathcal{N}_v$ is one of them with generator
$$
\sigma_v \mapsto 
\begin{pmatrix}
   0 & 0\\
   0 & 0
\end{pmatrix}, \quad
\tau_v \mapsto 
\begin{pmatrix}
   0 & 1\\
   0 & 0
\end{pmatrix}.
$$
Assume the ramified line $h^{(v)}|_{G_v}$ is equidistributed among the $p$ ramified lines in $H^1(G_v, \Ad).$
Then the probability that 
$$
H^1_\mathcal{N}(G_{S\cup\{v\}}, \Ad) = 0
$$
is $(p-1)/p$.
\end{theorem}

Before giving a proof of Theorem \ref{(p-1)/p}, 
the distribution of lines in the plane is interpreted in the language of measure theoretic probability in the following remark. If $(\Omega, P)$ is a probability space,
$E$ a set endowed with a $\sigma$-algebra, 
and $X:\Omega \to E$ a random variable (a measurable map), the probability distribution of $X$ is defined as the image measure of $P$ under $X$ on $E$ and is  
denoted as $X(P)$. 

\begin{remark}[The probability distribution of lines in the projective plane] \label{X}
For each $v \in C^{(p)}$, there is an $\F_p$-linear isomorphism 
$$
H^1(G_v, \Ad) \simeq \F_p^2.
$$
We may choose a bijection
$$
\varphi_v : \mathbb{P} (H^1(G_v, \Ad)) \longrightarrow \mathbb{P} (\F_p^2)
$$
such that 
$$
\varphi_v ( \{ H^1_{\operatorname{unr}}(G_v, \Ad)\} ) = (0:1),
$$
and
$$
\varphi_v ( \{ \mathcal{N}_v \} ) = (1:0),
$$ 
where $\{ H^1_{\operatorname{unr}}(G_v, \Ad)\} \in \mathbb{P} (H^1(G_v, \Ad))$ is the unramified line. 

As a concrete probability space, one could take 
$\Omega = C^{(p)}$ endowed  with a probability measure $P$.  
By defining the random variable  
\begin{IEEEeqnarray*}{rCl}
X : \Omega  & \longrightarrow & \mathbb{P}(\F_p^2)\setminus \{ (0:1) \} \\
v & \mapsto & \varphi_v \circ \operatorname{span}_{\F_p} h^{(v)}|_{G_v},
\end{IEEEeqnarray*}
the assumption in Theorem \ref{(p-1)/p} can be stated as follows:
The image measure $$X(P)= \varphi_v \circ \operatorname{span}_{\F_p} h^{(v)}|_{G_v}(P)$$ is the uniform distribution on 
$\mathbb{P}(\F_p^2)\setminus \{ (0:1) \}$; in particular, 
the probability that the Selmer rank increases by one is the measure of the set 
$$ 
\{ v \in \Omega: \dim H^1_\mathcal{N}(G_{S\cup\{v\}}, \Ad) = \dim H^1_\mathcal{N}(G_{S}, \Ad)  + 1 \},
$$
which is
\begin{IEEEeqnarray*}{rCl}
P ( \dim H^1_\mathcal{N}(G_{S\cup\{v\}}, \Ad) &=& \dim H^1_\mathcal{N}(G_{S}, \Ad)  + 1 ) \\
&=& P( \operatorname{span}_{\F_p} h^{(v)}|_{G_v} = \mathcal{N}_v ) \\
& = & P ( \varphi_v  \circ  \operatorname{span}_{\F_p} h^{(v)}|_{G_v} = \varphi_v(\mathcal{N}_v)) \\
& = & P( X = (1:0) ) \\
& = & X(P)((1:0)) \\   
&=&  1/p.
\end{IEEEeqnarray*}
\end{remark}

\begin{proof}[Proof of Theorem \ref{(p-1)/p}]
By Lemma \ref{localranks5} (if $p \geqslant 5$)
and Lemma \ref{localranks3} (if $p=3$),
$$ \dim_{\F_p} H^1(G_v, \Ad) =2,$$ 
so there are 
$$
\frac{\operatorname{card}(\F_p^2\setminus \{0\}) }{\operatorname{card}(\F_p^\times)} = \frac{p^2-1}{p-1} = p+1
$$ 
lines in $H^1(G_v, \Ad)$. One of the $p+1$ lines is unramified, since
$$
\dim_{\F_p} H^1_{\operatorname{unr}}(G_v, \Ad) =
\dim_{\F_p} H^0(G_v, \Ad) = 1,
$$
and one of the other $p+1$ lines spans $\mathcal{N}_v$.
By Lemma \ref{kernel}, there is one global cohomology class $h^{(v)}$ in the kernel of the restriction map
$$ 
H^1(G_{S\cup\{v\}}, \Ad) 
\longrightarrow
\bigoplus_{q\in S} H^1(G_q, \Ad)/\mathcal{N}_q,
$$
and $h^{(v)}$ is ramified at $v$. Hence $h^{(v)}|_{G_v}$ 
will span one of the $p$ ramified lines in $H^1(G_v, \Ad).$
The probability that 
$h^{(v)}|_{G_v}$
spans the line $\mathcal{N}_v$ is therefore $1/p,$
while the probability that 
$h|_{G_v} \not \in \mathcal{N}_v$
is $(p-1)/p.$
Since  $h|_{G_v} \not \in \mathcal{N}_v$ is equivalent to 
$$
H^1_\mathcal{N}(G_{S\cup\{v\}}, \Ad) = 0,
$$
it follows that level-raising at $v$ occurs with probability 
$$
\frac{p-1}{p}.
$$
\end{proof}
Jointly with Ravi Ramakrishna, we make the following conjecture.
\begin{conjecture}
In Remark \ref{X}, the random variable $X$ was defined, describing the distribution over primes of lines in the projective plane. These lines, parametrized by auxiliary primes, govern level-raising of $p$-adic even representations.
We conjecture that the random variable $X$ is equidistributed.
By Theorem \ref{(p-1)/p}, an affirmative answer to this conjecture implies that the density of primes raising the level of even $p$-adic representations is $(p-1)/p$.
The tables \ref{table:ell=163}, \ref{table:ell=277}, and \ref{table:ell=349} suggest so for $p=3$. However, Chebotarev's theorem does most likely not apply in this setting.
\end{conjecture}

\bibliographystyle{amsplain} 
\bibliography{references.bib}

\nocite{*}

\end{document}